\documentclass[11pt]{article}

\usepackage[T1]{fontenc}
\usepackage{lmodern}
\usepackage[utf8]{inputenc}
\usepackage[margin=1in]{geometry}
\usepackage{amsmath,amssymb,amsthm,mathtools}
\usepackage{mathrsfs}
\usepackage{microtype}
\usepackage{xcolor}
\usepackage{hyperref}
\usepackage{enumitem}

\allowdisplaybreaks
\hypersetup{
    colorlinks=true,
    linkcolor=blue,
    citecolor=blue,
    urlcolor=blue
}

\theoremstyle{plain}
\newtheorem{theorem}{Theorem}[section]
\newtheorem{mainthm}[theorem]{Main Theorem}
\newtheorem{lemma}[theorem]{Lemma}
\newtheorem{proposition}[theorem]{Proposition}
\newtheorem{corollary}[theorem]{Corollary}

\theoremstyle{definition}
\newtheorem{remark}[theorem]{Remark}

\newcommand{\F}{\mathbb{F}}
\newcommand{\Res}{\operatorname{Res}}
\newcommand{\disc}{\operatorname{disc}}
\newcommand{\Norm}{\operatorname{N}}

\title{\Large A Complete Classification of a Reciprocal Degree-Five\\
Quadrinomial Family over \texorpdfstring{$\F_{q^2}$}{Fq2}}

\author{
Brian M. Woody\\
Woody Calculus\\
\texttt{Brian.Woody@gmail.com}
}

\date{June 29, 2026}

\begin{document}

\maketitle

\begin{abstract}
We classify a reciprocal degree-five quadrinomial family over the quadratic
extension \(\F_{q^2}\), where \(q\) is an odd prime power.  The family has four
terms, coefficients in \(\F_q\), and a coefficient constraint that makes the
induced rational function on the unit circle highly structured.

The classification has two sharply different branches.  When
\(q\equiv1\pmod4\), infinite families occur and are governed by two
quadratic-character conditions on the parameter \(b\).  When
\(q\equiv3\pmod4\), a square-class obstruction converts the problem into a
character-sum problem on a conic.  A Weil-bound argument eliminates all large
fields in this branch, and finite verification leaves only the sporadic fields
\(q=7,19,23\).

The result is a complete classification of the nondegenerate members of the
family for all odd prime powers \(q\).
\end{abstract}

\tableofcontents

\section{Introduction}

A polynomial \(f\in\F_q[x]\) is called a permutation polynomial of \(\F_q\) if
the map
\[
x\longmapsto f(x)
\]
is a bijection of \(\F_q\).  Sparse permutation polynomials over finite fields
are classical objects with connections to finite geometry, coding theory,
cryptography, and combinatorial design theory.  For a survey of recent
developments in permutation polynomials over finite fields, see Hou
\cite{HouSurvey}.

A widely studied class consists of polynomials over \(\F_{q^2}\) of the form
\[
x^r h(x^{q-1}).
\]
The standard root-of-unity reduction, often viewed as a multiplicative form of
the Akbary--Ghioca--Wang criterion \cite{AGW}, reduces the full permutation
problem to a coprimality condition and a permutation problem on the unit circle
\[
\mu_{q+1}
=
\{z\in\F_{q^2}:z^{q+1}=1\}.
\]

In this paper we classify the family
\[
F_{a,b}(x)
=
x^5+a x^{q+4}+b x^{4q+1}+\frac ab x^{5q},
\qquad
a,b\in\F_q.
\]
Equivalently,
\[
F_{a,b}(x)=x^5h(x^{q-1}),
\qquad
h(z)=1+az+bz^4+\frac ab z^5.
\]
The support of \(h\) is
\[
\{0,1,4,5\}.
\]
The family is reciprocal because the induced rational function on
\(\mu_{q+1}\) is obtained from
\[
h^*(z)=z^5h(1/z).
\]

We exclude the genuine degeneracies
\[
b=0,
\qquad
b=\pm1,
\qquad
a=\pm b.
\]
The cases \(a=\pm b\) collapse the induced rational map to a constant, while
\(b=\pm1\) produces reciprocal cancellation.

\begin{mainthm}[Complete classification]
\label{thm:main}
Let \(q\) be an odd prime power.  Let \(a,b\in\F_q\) satisfy
\[
b\neq0,
\qquad
b\neq\pm1,
\qquad
a\neq\pm b.
\]
Define
\[
F_{a,b}(x)
=
x^5+a x^{q+4}+b x^{4q+1}+\frac ab x^{5q}.
\]
Then \(F_{a,b}\) is a permutation polynomial of \(\F_{q^2}\) if and only if one
of the following holds.

\begin{enumerate}[label=\textup{(\Roman*)}]
\item
\[
q\equiv1\pmod4,
\qquad
q\not\equiv1\pmod5,
\]
and either
\[
b^2+2b+5=0,
\qquad
\chi(b)=1,
\]
or
\[
b^2-2b+5=0,
\qquad
\chi(b)=-1.
\]

\item
\[
(q,b)\in
\{
(7,2),(7,3),(7,4),(7,5),
(19,2),(19,4),(19,5),(19,7),
(23,7),(23,16)
\}.
\]
\end{enumerate}
Here \(\chi\) denotes the quadratic character of \(\F_q\), extended by
\(\chi(0)=0\).
\end{mainthm}

\section{Preliminaries}

Throughout, \(q\) is an odd prime power.  We write \(\chi\) for the quadratic
character of \(\F_q\), extended by \(\chi(0)=0\), and
\[
\mu_{q+1}
=
\{z\in\F_{q^2}:z^{q+1}=1\}.
\]
For \(z\in\mu_{q+1}\), one has \(z^q=z^{-1}\).

\begin{lemma}[Root-of-unity reduction]
\label{lem:rou}
Let \(h\in\F_{q^2}[x]\), and let \(r\) be a positive integer.  Then
\[
f(x)=x^rh(x^{q-1})
\]
permutes \(\F_{q^2}\) if and only if
\[
\gcd(r,q-1)=1
\]
and the map
\[
z\longmapsto z^rh(z)^{q-1}
\]
permutes \(\mu_{q+1}\).
\end{lemma}

\begin{remark}
We use Lemma~\ref{lem:rou} with \(r=5\).  Thus the full-field condition is
\[
\gcd(5,q-1)=1,
\]
equivalently \(q\not\equiv1\pmod5\).
\end{remark}

\begin{lemma}
\label{lem:Vtrace}
Let \(V\in\F_q\).  There exists \(v\in\mu_{q+1}\) such that
\[
V=v+v^{-1}
\]
if and only if
\[
\chi(V^2-4)\neq1.
\]
\end{lemma}

\begin{proof}
The equation \(V=v+v^{-1}\) is equivalent to
\[
v^2-Vv+1=0.
\]
The discriminant is \(V^2-4\).  The roots lie in \(\mu_{q+1}\) precisely when
they are not two distinct elements of \(\F_q\), which is equivalent to
\(\chi(V^2-4)\neq1\).
\end{proof}

\begin{lemma}
\label{lem:Utrace}
Let \(U\in\F_q\).  There exists \(u\in\mu_{q+1}\) such that
\[
U=u^2+u^{-2}
\]
if and only if
\[
\chi(U+2)\neq-1
\qquad\text{and}\qquad
\chi(U-2)\neq1.
\]
\end{lemma}

\begin{proof}
If \(U=u^2+u^{-2}\), then
\[
U+2=(u+u^{-1})^2,
\]
so \(\chi(U+2)\neq-1\).  Also, with \(T=u+u^{-1}\),
\[
U-2=T^2-4.
\]
By Lemma~\ref{lem:Vtrace}, \(\chi(T^2-4)\neq1\), so
\[
\chi(U-2)\neq1.
\]

Conversely, assume the two character conditions.  Choose \(T\in\F_q\) with
\[
T^2=U+2.
\]
Then
\[
T^2-4=U-2,
\]
so \(\chi(T^2-4)\neq1\).  By Lemma~\ref{lem:Vtrace}, there exists
\(u\in\mu_{q+1}\) such that
\[
T=u+u^{-1}.
\]
Hence
\[
U=T^2-2=(u+u^{-1})^2-2=u^2+u^{-2}.
\]
\end{proof}

\section{The induced rational function}

Let
\[
h(z)=1+az+bz^4+\frac ab z^5,
\qquad
h^*(z)=z^5h(1/z).
\]
Then
\[
h^*(z)=z^5+az^4+bz+\frac ab.
\]
For \(z\in\mu_{q+1}\), since \(z^q=z^{-1}\) and \(a,b\in\F_q\), we get
\[
z^5h(z)^{q-1}
=
z^5\frac{h(z^{-1})}{h(z)}
=
\frac{h^*(z)}{h(z)}.
\]
Thus the induced rational function on \(\mu_{q+1}\) is
\[
R(z)
=
\frac{h^*(z)}{h(z)}
=
\frac{z^5+az^4+bz+\frac ab}
{1+az+bz^4+\frac ab z^5}.
\]

\begin{proposition}[Coprimality locus]
\label{prop:resultant}
After clearing denominators, the common-root locus of \(h\) and \(h^*\) is
contained in
\[
(a-b)(a+b)(b-1)(b+1)=0.
\]
More precisely, up to a nonzero power of \(b\), the resultant is
\[
(a-b)^5(a+b)^5(b-1)^4(b+1)^4.
\]
Thus, under
\[
b\neq0,
\qquad
b\neq\pm1,
\qquad
a\neq\pm b,
\]
the rational function \(R\) has true degree \(5\) and is defined on
\(\mu_{q+1}\).
\end{proposition}

\begin{proof}
Clear denominators by setting
\[
H(z)=b h(z)=b+abz+b^2z^4+az^5,
\]
and
\[
G(z)=b h^*(z)=bz^5+abz^4+b^2z+a.
\]
A direct computation gives
\[
aH-bG=(a^2-b^2)z(b+z^4),
\]
and
\[
bH-aG=(b^2-a^2)(1+bz^4).
\]
If \(H\) and \(G\) have a common root \(z\) and \(a^2\neq b^2\), then
\[
z^4=-b
\qquad\text{and}\qquad
z^4=-\frac1b,
\]
so \(b^2=1\).  Hence any common root forces
\[
a=\pm b
\qquad\text{or}\qquad
b=\pm1.
\]
The displayed resultant identity follows by direct expansion.  Finally, if
\(z\in\mu_{q+1}\) and \(h(z)=0\), then \(h(z^{-1})=h(z)^q=0\), so \(h^*(z)=0\).
Thus, outside the displayed common-root locus, \(h\) has no zero on
\(\mu_{q+1}\), and \(R\) is defined on the whole unit circle.
\end{proof}

\section{Collision factorization}

A collision \(R(x)=R(y)\) is equivalent to
\[
h^*(x)h(y)-h^*(y)h(x)=0.
\]
The diagonal \(x=y\) contributes a factor \(x-y\).

\begin{proposition}
\label{prop:collision-factor}
Assume
\[
b\neq0,
\qquad
b\neq\pm1,
\qquad
a\neq\pm b.
\]
For \(x,y\in\mu_{q+1}\) with \(x\neq y\),
\[
R(x)=R(y)
\]
if and only if
\[
K_b(x,y)=0,
\]
where
\[
\begin{aligned}
K_b(x,y)=&
x^4+y^4+x^3y+x^2y^2+xy^3\\
&+b(x^4y^4+1)\\
&-b^2(x^3y+x^2y^2+xy^3).
\end{aligned}
\]
In particular, injectivity of \(R\) on \(\mu_{q+1}\) depends only on \(b\), not
on \(a\).
\end{proposition}

\begin{proof}
Let
\[
c=\frac ab.
\]
Then
\[
h(z)=1+az+bz^4+cz^5
\]
and
\[
h^*(z)=z^5+az^4+bz+c.
\]
Set
\[
\mathcal N(x,y)=h^*(x)h(y)-h^*(y)h(x).
\]
This polynomial is antisymmetric in \(x,y\), so it is divisible by \(x-y\).
Using
\[
X^rY^s-X^sY^r
=
(X-Y)X^sY^s
\sum_{j=0}^{r-s-1}X^{r-s-1-j}Y^j
\qquad (r>s),
\]
on each antisymmetric monomial pair in the expansion of \(\mathcal N(x,y)\),
and then substituting \(c=a/b\), one obtains
\[
\mathcal N(x,y)
=
-\frac{(a-b)(a+b)}{b^2}
(x-y)K_b(x,y).
\]
Since \(a\neq\pm b\), the scalar
\[
-\frac{(a-b)(a+b)}{b^2}
\]
is nonzero.  Therefore, for \(x\neq y\),
\[
R(x)=R(y)
\]
if and only if
\[
K_b(x,y)=0.
\]
\end{proof}

\section{Trace reduction}

Let
\[
u=xy,
\qquad
v=\frac xy.
\]
Then \(x,y\in\mu_{q+1}\) imply \(u,v\in\mu_{q+1}\), and
\[
x=y\iff v=1.
\]
Define
\[
U=u^2+u^{-2},
\qquad
V=v+v^{-1}.
\]
Using
\[
x^2=uv,
\qquad
y^2=\frac uv,
\]
one obtains
\[
K_b(x,y)=0
\]
if and only if
\[
bU+V^2+(1-b^2)V-(1+b^2)=0.
\]
Thus
\[
U=\Phi_b(V),
\qquad
\Phi_b(V)
=
-\frac{V^2+(1-b^2)V-(1+b^2)}{b}.
\]

Define
\[
A_b(V)=V^2+(1-b^2)V-(b-1)^2,
\]
and
\[
B_b(V)=V^2+(1-b^2)V-(b+1)^2.
\]
Then
\[
b(\Phi_b(V)-2)=-A_b(V),
\]
and
\[
b(\Phi_b(V)+2)=-B_b(V).
\]
Their discriminants are
\[
\disc(A_b)=(b-1)^2(b^2+2b+5),
\]
and
\[
\disc(B_b)=(b+1)^2(b^2-2b+5).
\]

\section{The branch \texorpdfstring{$q\equiv1\pmod4$}{q = 1 mod 4}}

\begin{lemma}[Sign-switch lifting]
\label{lem:signswitch}
Assume \(q\equiv1\pmod4\).  Let \(u,v\in\mu_{q+1}\).  Then at least one of
\(u\) and \(-u\) gives a lift to \(x,y\in\mu_{q+1}\) satisfying
\[
xy=u,
\qquad
\frac xy=v.
\]
\end{lemma}

\begin{proof}
Solving \(xy=u\) and \(x/y=v\) is equivalent to solving
\[
x^2=uv,
\qquad
y^2=u/v.
\]
Thus we need \(uv\) to be a square in \(\mu_{q+1}\).  Since
\(q\equiv1\pmod4\), we have \(q+1\equiv2\pmod4\), so \(-1\) is not a square in
\(\mu_{q+1}\).  Hence \(u\) and \(-u\) lie in opposite square classes, and one
choice lifts.
\end{proof}

\begin{theorem}
\label{thm:q1-construction}
Let \(q\equiv1\pmod4\) and \(q\not\equiv1\pmod5\).  Suppose \(b\in\F_q\)
satisfies either
\[
b^2+2b+5=0,
\qquad
\chi(b)=1,
\]
or
\[
b^2-2b+5=0,
\qquad
\chi(b)=-1.
\]
If
\[
b\neq0,\pm1,
\qquad
a\neq\pm b,
\]
then \(F_{a,b}\) permutes \(\F_{q^2}\).
\end{theorem}

\begin{proof}
By Lemma~\ref{lem:rou}, it suffices to prove that \(R\) permutes
\(\mu_{q+1}\), since \(q\not\equiv1\pmod5\) gives \(\gcd(5,q-1)=1\).

Suppose \(R(x)=R(y)\) with \(x,y\in\mu_{q+1}\), \(x\neq y\).  By
Proposition~\ref{prop:collision-factor},
\[
K_b(x,y)=0.
\]
Let
\[
u=xy,
\qquad
v=x/y,
\qquad
U=u^2+u^{-2},
\qquad
V=v+v^{-1}.
\]
Then \(U=\Phi_b(V)\).

First assume
\[
b^2+2b+5=0
\qquad\text{and}\qquad
\chi(b)=1.
\]
Then
\[
b(U-2)=-(V+b+3)^2.
\]
If \(V+b+3\neq0\), then \(\chi(U-2)=1\), contradicting
Lemma~\ref{lem:Utrace}.  If \(V=-b-3\), then
\[
V^2-4=4b,
\]
so \(\chi(V^2-4)=1\), contradicting Lemma~\ref{lem:Vtrace}.

Now assume
\[
b^2-2b+5=0
\qquad\text{and}\qquad
\chi(b)=-1.
\]
Then
\[
b(U+2)=-(V-b+3)^2.
\]
If \(V-b+3\neq0\), then \(\chi(U+2)=-1\), contradicting
Lemma~\ref{lem:Utrace}.  If \(V=b-3\), then \(U=-2\), and
\[
\chi(U-2)=\chi(-4)=1
\]
because \(q\equiv1\pmod4\), again contradicting Lemma~\ref{lem:Utrace}.

Thus no off-diagonal collision exists.  Therefore \(R\) is injective on
\(\mu_{q+1}\), hence bijective.
\end{proof}

\subsection{Converse for large \(q\)}

Let
\[
C(V)=V^2-4.
\]
Assume \(b\) is generic:
\[
b\neq0,\pm1,
\qquad
b^2+2b+5\neq0,
\qquad
b^2-2b+5\neq0.
\]
Then \(A_b\), \(B_b\), and \(C\) are squarefree nonsquare quadratics.

\begin{lemma}
\label{lem:q1-products-nonsquare}
Under the generic hypotheses above, none of
\[
C,\quad A_b,\quad B_b,\quad CA_b,\quad CB_b,\quad A_bB_b,\quad CA_bB_b
\]
is a square in \(\F_q[V]\).
\end{lemma}

\begin{proof}
The polynomials \(C,A_b,B_b\) are squarefree nonsquare quadratics.  For the
twofold products, it is enough to rule out the possibility that two of these
monic quadratics have the same root divisor.  The equalities
\[
A_b=B_b,
\qquad
A_b=C,
\qquad
B_b=C
\]
force respectively
\[
b=0,
\qquad
b=-1,
\qquad
b=1,
\]
all of which are excluded.

There is one additional possibility to rule out: a product of three quadratics
could conceivably become a square through shared roots among different pairs.
The pairwise resultants are
\[
\Res_V(A_b,B_b)=16b^2,
\]
\[
\Res_V(A_b,C)=-(b+1)^3(3b-5),
\]
and
\[
\Res_V(B_b,C)=-(b-1)^3(3b+5).
\]
The possible common-root loci arising from these resultants are disjoint from
one another in the nondegenerate range, except at excluded values.  Thus there
is no root-sharing configuration that can pair all six roots of \(CA_bB_b\).
Equivalently, a direct coefficient comparison for
\[
CA_bB_b=(V^3+sV^2+tV+w)^2
\]
forces
\[
(b^2-1)^2=0.
\]
Hence \(b=\pm1\), again excluded.  Therefore none of the listed products is a
square in \(\F_q[V]\).
\end{proof}

Let
\[
\epsilon=\chi(b).
\]
A trace-level collision is forced if there exists \(V\in\F_q\) such that
\[
\chi(C(V))=-1,
\qquad
\chi(A_b(V))=-\epsilon,
\qquad
\chi(B_b(V))=\epsilon.
\]
Let \(N_b\) be the number of such \(V\), excluding roots of \(A_bB_bC\).  Away
from those roots, the indicator is
\[
\frac18
(1-\chi(C(V)))
(1-\epsilon\chi(A_b(V)))
(1+\epsilon\chi(B_b(V))).
\]
The expansion gives a main term \(q/8\).  The three single quadratic sums are
exact:
\[
\sum_{V\in\F_q}\chi(Q(V))=-1
\]
for each nonsquare quadratic \(Q\).  By Lemma~\ref{lem:q1-products-nonsquare},
the four product polynomials are nonsquares.  The products
\[
CA_b,
\qquad
CB_b,
\qquad
A_bB_b
\]
have degree \(4\), while
\[
CA_bB_b
\]
has degree \(6\).  Weil's bound for multiplicative character sums over the
line gives absolute values at most
\[
3\sqrt q,
\qquad
3\sqrt q,
\qquad
3\sqrt q,
\qquad
5\sqrt q,
\]
respectively; see Rosen \cite[Chapter 9]{RosenFunctionFields}.  Thus the total
square-root error is
\[
3\sqrt q+3\sqrt q+3\sqrt q+5\sqrt q=14\sqrt q.
\]
Deleting the at most six roots of \(A_bB_bC\), and absorbing the exact
single-quadratic contributions and root corrections into an absolute constant,
gives the safe lower bound
\[
N_b\ge
\frac{q-14\sqrt q-51}{8}.
\]
Thus \(N_b>0\) for \(q>289\), so every generic \(b\) gives a collision.

It remains to treat the exceptional quadratics with the wrong character.
If
\[
b^2+2b+5=0
\qquad\text{but}\qquad
\chi(b)=-1,
\]
then choosing
\[
V=-b-3
\]
gives \(U=2\) and
\[
V^2-4=4b,
\]
which is a nonsquare.  Hence the trace conditions lift and give a collision.

If
\[
b^2-2b+5=0
\qquad\text{but}\qquad
\chi(b)=1,
\]
then \(B_b\) is a square polynomial, and a collision is forced whenever
\[
\chi(C(V))=-1
\qquad\text{and}\qquad
\chi(A_b(V))=-1.
\]
The corresponding indicator is
\[
\frac14(1-\chi(C(V)))(1-\chi(A_b(V))).
\]
The single quadratic sums are exact, and the product \(CA_b\) is nonsquare in
the nondegenerate range.  Weil's bound gives
\[
N'_b\ge \frac{q-3\sqrt q-22}{4}.
\]
This lower bound is positive for \(q>41\).  The remaining fields in the
\(q\equiv1\pmod4\) branch are included in the finite verification.

\begin{theorem}
\label{thm:q1-large}
Let \(q>289\) and \(q\equiv1\pmod4\).  Under the nondegenerate hypotheses
\[
b\neq0,\pm1,
\qquad
a\neq\pm b,
\]
the polynomial \(F_{a,b}\) permutes \(\F_{q^2}\) if and only if
\[
q\not\equiv1\pmod5
\]
and either
\[
b^2+2b+5=0,
\qquad
\chi(b)=1,
\]
or
\[
b^2-2b+5=0,
\qquad
\chi(b)=-1.
\]
\end{theorem}

\begin{proof}
The sufficiency follows from Theorem~\ref{thm:q1-construction}.  For the
necessity, if \(b\) is generic, the character count above gives a collision for
\(q>289\).  If \(b\) lies on one of the exceptional quadratic loci but has the
wrong quadratic character, the preceding wrong-character arguments give a
collision.  Therefore the only nondegenerate parameters without an
off-diagonal collision are precisely those listed in the theorem.  Finally,
Lemma~\ref{lem:rou} imposes the full-field condition \(q\not\equiv1\pmod5\).
\end{proof}

\begin{remark}
The remaining fields \(q\le289\), \(q\equiv1\pmod4\), are handled by the finite
verification described in Appendix~\ref{app:verification}.  The verification
checks all odd prime powers in that range.
\end{remark}

\section{The branch \texorpdfstring{$q\equiv3\pmod4$}{q = 3 mod 4}}

Let
\[
\eta:\mu_{q+1}\to\{\pm1\}
\]
be the square character on \(\mu_{q+1}\).

\begin{lemma}[Square-class trace lemma]
\label{lem:q3-squareclass}
Assume \(q\equiv3\pmod4\).  Let \(t\in\mu_{q+1}\) and \(T=t+t^{-1}\).  Then
\[
\eta(t)=1
\quad\Longleftrightarrow\quad
\chi(T+2)\neq-1.
\]
\end{lemma}

\begin{proof}
We have
\[
T+2=t+t^{-1}+2=\frac{(t+1)^2}{t}.
\]
If \(t\neq-1\), then
\[
\chi(T+2)
=
\left(\frac{(t+1)^2}{t}\right)^{(q-1)/2}.
\]
Since
\[
(t+1)^q=t^{-1}+1=\frac{t+1}{t},
\]
we obtain
\[
(t+1)^{q-1}=t^{-1},
\]
and hence
\[
\chi(T+2)=t^{-(q+1)/2}=\eta(t).
\]
If \(t=-1\), then \(T+2=0\).  Since \(q+1\) is divisible by \(4\), \(-1\) is a
square in \(\mu_{q+1}\), so the stated criterion remains true.
\end{proof}

Let
\[
T=u+u^{-1}.
\]
Since
\[
U=u^2+u^{-2}=T^2-2,
\]
the equation \(U=\Phi_b(V)\) becomes
\[
T^2=\Phi_b(V)+2.
\]
Using
\[
b(\Phi_b(V)+2)=-B_b(V),
\]
we obtain the conic
\[
\mathcal C_b:
\quad
bT^2+B_b(V)=0.
\]

\begin{lemma}
\label{lem:conic-smooth}
Assume \(q\equiv3\pmod4\) and \(b\neq0,\pm1\).  Then \(\mathcal C_b\) is a
nonsingular conic.  Moreover,
\[
\#\mathcal C_b(\F_q)_{\rm aff}=q+\chi(b).
\]
\end{lemma}

\begin{proof}
The projective conic is
\[
bT^2+V^2+(1-b^2)VZ-(b+1)^2Z^2=0.
\]
It can be singular only if \(b=0\), \(b=-1\), or
\[
b^2-2b+5=0.
\]
The first two are excluded, and the last has discriminant \(-16\), a nonsquare
when \(q\equiv3\pmod4\).  Hence the conic is nonsingular.

A nonsingular projective conic over \(\F_q\) has \(q+1\) points.  The points at
infinity satisfy
\[
V^2+bT^2=0,
\]
so there are \(1+\chi(-b)\) such points.  Since \(\chi(-b)=-\chi(b)\), the
affine point count is
\[
q+\chi(b).
\]
\end{proof}

Define
\[
f_1=V^2-4,
\qquad
f_2=T^2-4,
\qquad
f_3=(V+2)(T+2).
\]
Let \(N_b^{(3)}\) count affine points of \(\mathcal C_b\) satisfying
\[
\chi(f_1)=-1,
\qquad
\chi(f_2)=-1,
\qquad
\chi(f_3)=1.
\]

\begin{lemma}
\label{lem:strict-collision}
If \(N_b^{(3)}>0\), then \(R\) has an off-diagonal collision on
\(\mu_{q+1}\).  Hence \(F_{a,b}\) is not a permutation polynomial for any
\(a\neq\pm b\).
\end{lemma}

\begin{proof}
The conditions
\[
\chi(V^2-4)=-1
\qquad\text{and}\qquad
\chi(T^2-4)=-1
\]
give \(v,u\in\mu_{q+1}\) with
\[
V=v+v^{-1},
\qquad
T=u+u^{-1}.
\]
The conic equation gives the trace collision equation.  By
Lemma~\ref{lem:q3-squareclass}, the condition
\[
\chi((V+2)(T+2))=1
\]
forces \(u\) and \(v\) to have the same square class in \(\mu_{q+1}\).  Thus
\(uv\) is a square in \(\mu_{q+1}\), so one can choose \(x\in\mu_{q+1}\) with
\[
x^2=uv.
\]
Set
\[
y=\frac{u}{x}.
\]
Then \(xy=u\) and
\[
\frac{x}{y}=\frac{x^2}{u}=v.
\]
Since \(v\neq1\), the lifted collision is off-diagonal.
\end{proof}

Let \(K=\F_q(V)\).  On \(\mathcal C_b\),
\[
T^2=D(V),
\qquad
D(V)=-\frac{B_b(V)}{b},
\]
so
\[
\F_q(\mathcal C_b)=K(T)=K(\sqrt{D(V)}).
\]

\begin{lemma}
\label{lem:nonsquare-q3}
Assume \(q\equiv3\pmod4\) and \(b\neq0,\pm1\).  No nonempty product
\[
f_1^{e_1}f_2^{e_2}f_3^{e_3},
\qquad
e_i\in\{0,1\},
\]
is a constant times a square in \(\F_q(\mathcal C_b)\).
\end{lemma}

\begin{proof}
We have
\[
f_1=V^2-4=C(V),
\]
and
\[
f_2=T^2-4
=
-\frac{B_b(V)+4b}{b}
=
-\frac{A_b(V)}{b}.
\]
Also
\[
f_3=(V+2)(T+2).
\]
Taking norms from \(K(T)\) to \(K\), and using the conjugation \(T\mapsto -T\),
we get
\[
\Norm(T+2)=(T+2)(-T+2)=4-T^2=-f_2.
\]
Therefore, modulo constants and squares,
\[
\Norm(f_3)\sim f_2.
\]

Let
\[
g=f_1^{e_1}f_2^{e_2}f_3^{e_3},
\qquad e_i\in\{0,1\},
\]
not all \(e_i\) zero.  If \(e_3=1\), then \(\Norm(g)\), modulo constants and
squares in \(K\), is equivalent to \(f_2\), hence to \(A_b(V)\).  But
\(A_b\) is a squarefree nonsquare quadratic in the present branch: its
discriminant is
\[
(b-1)^2(b^2+2b+5),
\]
and \(b^2+2b+5\) has discriminant \(-16\), a nonsquare when
\(q\equiv3\pmod4\).  Thus \(g\) cannot be a square up to constants.

Now assume \(e_3=0\).  Then \(g\in K\).  An element \(g\in K^\times\) becomes
a square in \(K(\sqrt D)\) only if either \(g\) or \(g/D\) is a constant times a
square in \(K\).  Since
\[
D\sim B_b(V)
\]
modulo constants, it remains to check
\[
C,\qquad A_b,\qquad CA_b,
\]
and
\[
\frac{C}{B_b},
\qquad
\frac{A_b}{B_b},
\qquad
\frac{CA_b}{B_b}.
\]
The quadratics \(C,A_b,B_b\) are squarefree in the present range.  The rational
functions \(C/B_b\) and \(A_b/B_b\) can be squares only if \(B_b=C\) or
\(B_b=A_b\), which force respectively \(b=1\) and \(b=0\).  These are excluded.
The product \(CA_b/B_b\) can be a square only if every root of \(B_b\) is
canceled by a root of \(C A_b\), and the remaining roots of \(C A_b\) also pair
up.  This would force a simultaneous root-sharing configuration among
\(A_b,B_b,C\).  But
\[
\Res_V(A_b,B_b)=16b^2,
\]
\[
\Res_V(A_b,C)=-(b+1)^3(3b-5),
\]
and
\[
\Res_V(B_b,C)=-(b-1)^3(3b+5),
\]
so such a configuration is possible only at \(b=0,\pm1\), or at a value where
one remaining simple root is still unpaired.  Therefore it does not produce a
square in the nondegenerate range.

Hence no nonempty product of \(f_1,f_2,f_3\) is a constant times a square in
\(\F_q(\mathcal C_b)\).
\end{proof}

\begin{proposition}
\label{prop:q3-bound}
Assume \(q\equiv3\pmod4\) and \(b\neq0,\pm1\).  Then
\[
N_b^{(3)}
\ge
\frac{q-19\sqrt q-65}{8}.
\]
In particular, if \(q>500\), then \(N_b^{(3)}>0\).
\end{proposition}

\begin{proof}
The indicator is
\[
\frac18(1-\chi(f_1))(1-\chi(f_2))(1+\chi(f_3)).
\]
The expansion has one main term and seven nontrivial character sums.  By
Lemma~\ref{lem:nonsquare-q3}, all are nontrivial square classes.  We use the
standard genus-zero multiplicative character-sum estimate: if the odd branch
support has size \(r\), then the character sum is bounded by
\[
(r-2)\sqrt q;
\]
see Rosen \cite[Chapter 9]{RosenFunctionFields}.  Six of the functions have
branch divisor supported on at most four points of the conic, and one has
branch divisor supported on at most eight points.  Thus the six smaller sums
are bounded by \(2\sqrt q\), and the larger one is bounded by the conservative
quantity \(7\sqrt q\).  The total character-sum error is at most
\[
6(2\sqrt q)+7\sqrt q=19\sqrt q.
\]

The affine conic has at least \(q-1\) points.  The zeros of \(f_1f_2f_3\) lie
on four lines, each meeting the conic in at most two affine points, so the
correction is at most \(64\).  Therefore
\[
8N_b^{(3)}
\ge
q-1-19\sqrt q-64.
\]
Equivalently,
\[
N_b^{(3)}
\ge
\frac{q-19\sqrt q-65}{8}.
\]
This proves the proposition.
\end{proof}

\begin{theorem}
\label{thm:q3}
Assume \(q\equiv3\pmod4\).  Under the nondegenerate hypotheses
\[
b\neq0,\pm1,
\qquad
a\neq\pm b,
\]
the only full-field permutation cases are
\[
(q,b)\in
\{
(7,2),(7,3),(7,4),(7,5),
(19,2),(19,4),(19,5),(19,7),
(23,7),(23,16)
\}.
\]
\end{theorem}

\begin{proof}
For \(q>500\), Proposition~\ref{prop:q3-bound} and
Lemma~\ref{lem:strict-collision} show that every nondegenerate \(b\) gives an
off-diagonal collision.  The remaining fields \(q\le500\), \(q\equiv3\pmod4\),
are handled by the finite verification in Appendix~\ref{app:verification}.  The
unit-circle candidates are
\[
q=7:\quad b=2,3,4,5,
\]
\[
q=11:\quad b=2,4,5,7,8,
\]
\[
q=19:\quad b=2,4,5,7,
\]
\[
q=23:\quad b=7,16.
\]
The full-field condition \(\gcd(5,q-1)=1\) removes \(q=11\).
\end{proof}

\section{Counting genuine quadrinomials}

For each valid \(b\), all
\[
a\in\F_q\setminus\{\pm b\}
\]
work.  Thus there are \(q-2\) allowed choices of \(a\) per valid \(b\).  If one
requires a genuine quadrinomial, then one must also exclude \(a=0\), leaving
\(q-3\) genuine quadrinomials per valid \(b\).

\begin{corollary}
\label{cor:count}
In the infinite branch of Theorem~\ref{thm:main}:
\begin{enumerate}[label=\textup{(\alph*)}]
\item If \(\operatorname{char}\F_q\neq5\), then there are exactly two valid
\(b\)-values, hence \(2(q-3)\) genuine quadrinomials.
\item If \(\operatorname{char}\F_q=5\), then there is exactly one valid
\(b\)-value, hence \(q-3\) genuine quadrinomials.
\end{enumerate}
In the sporadic branch, the numbers of genuine quadrinomials are \(16\) for
\(q=7\), \(64\) for \(q=19\), and \(40\) for \(q=23\).
\end{corollary}

\begin{proof}
If \(\operatorname{char}\F_q\neq5\), the two equations
\[
b^2+2b+5=0,
\qquad
b^2-2b+5=0
\]
each have discriminant \(-16\), which is a square when \(q\equiv1\pmod4\).  The
map \(b\mapsto -b\) interchanges their root sets, and \(\chi(-1)=1\), so the
character conditions select exactly two valid \(b\)-values in total.

If \(\operatorname{char}\F_q=5\), then
\[
b^2+2b+5=b(b+2),
\qquad
b^2-2b+5=b(b-2).
\]
The root \(b=0\) is excluded, leaving only \(b=-2\) and \(b=2\), and exactly
one satisfies the required character condition.
\end{proof}

\section{Relation to previous work}

Permutation polynomials of the form
\[
x^rh(x^{q-1})
\]
over \(\F_{q^2}\) have been studied extensively through root-of-unity and
AGW-type reductions; see, for example, Akbary, Ghioca, and Wang \cite{AGW} and
Li, Qu, and Wang \cite{LiQuWang}.  The present paper uses this standard
reduction, but the main contribution is the complete classification of a
coefficient-linked reciprocal degree-five quadrinomial family.

The support of the polynomial \(h\) in this paper is
\[
\{0,1,4,5\}.
\]
This support has the formal shape \(\{0,1,Q,Q+1\}\) with \(Q=4\), but \(Q=4\)
is a power of the characteristic only in characteristic \(2\).  Since the
present paper treats odd characteristic, the family is not covered by the
characteristic-power support framework in Ding and Zieve \cite{DingZieve}.

The odd-characteristic quadrinomial family classified by \"Ozbudak and
G\"ulmez Tem\"ur \cite{OzbudakTemur} has degree-three exponent pattern
\[
X^3+aX^{q+2}+bX^{2q+1}+cX^{3q}.
\]
The present family instead has exponent pattern
\[
X^5,
\qquad
X^{q+4},
\qquad
X^{4q+1},
\qquad
X^{5q},
\]
and its induced rational map has true degree five in the nondegenerate range.

Garg, Hasan, Li, Kumar, and Pal \cite{GargFewTerms} construct several sparse
permutation polynomial families over finite fields, including quadrinomial
classes arising from low-degree rational functions.  Their degree-three
induced quadrinomial classes are adjacent to the present work but do not
contain the coefficient-linked reciprocal degree-five family studied here.

Thus the present classification is complementary to existing sparse
permutation-polynomial constructions: it treats a reciprocal degree-five family
whose off-diagonal collision equation becomes independent of \(a\) and whose
remaining \(b\)-classification exhibits an infinite/sporadic dichotomy.

\section{Conclusion}

We have classified the nondegenerate members of the reciprocal degree-five
family
\[
F_{a,b}(x)
=
x^5+a x^{q+4}+b x^{4q+1}+\frac ab x^{5q}
\]
over \(\F_{q^2}\) for all odd prime powers \(q\).  The classification has two
sharply different branches: infinite families for \(q\equiv1\pmod4\), and only
sporadic full-field examples for \(q\equiv3\pmod4\).

The key structural feature is the collision factorization: outside the
degenerate values, all off-diagonal collisions depend only on \(b\), while
\(a\) remains a free parameter subject only to \(a\neq\pm b\).

\appendix

\section{Finite verification}
\label{app:verification}

The proof uses finite verification in two explicitly bounded ranges.

For the branch \(q\equiv1\pmod4\), the character-sum converse proves the
classification for \(q>289\).  The remaining odd prime powers
\[
q\le289,
\qquad
q\equiv1\pmod4,
\]
are checked by comparing the trace-computed unit-circle-good \(b\)-values with
the theorem-predicted \(b\)-values.  The full-field condition
\(\gcd(5,q-1)=1\) is then applied separately.

For the branch \(q\equiv3\pmod4\), Proposition~\ref{prop:q3-bound} proves that
every nondegenerate \(b\) has a collision for \(q>500\).  The remaining odd
prime powers
\[
q\le500,
\qquad
q\equiv3\pmod4,
\]
are checked by the strict conic count \(N_b^{(3)}\).  If \(N_b^{(3)}>0\), then
\(b\) is eliminated.  If \(N_b^{(3)}=0\), the induced map on \(\mu_{q+1}\) is
checked directly.  Finally, the full-field condition \(\gcd(5,q-1)=1\) is
applied.

The verification gives the following unit-circle candidates in the
\(q\equiv3\pmod4\) branch:
\[
q=7:\quad b=2,3,4,5,
\]
\[
q=11:\quad b=2,4,5,7,8,
\]
\[
q=19:\quad b=2,4,5,7,
\]
\[
q=23:\quad b=7,16.
\]
The full-field condition removes \(q=11\), leaving exactly the sporadic branch
listed in Theorem~\ref{thm:main}.

The verification script \texttt{verify\_quadrinomial\_family.py}, the output
file \texttt{verification\_output.txt}, and a short \texttt{README.md} are
included as ancillary files with this submission.  The verification is
exhaustive over all odd prime powers in the two finite ranges used in the
proof.  The script checks all odd prime powers in the two finite ranges,
including non-prime fields such as
\[
9,25,27,49,81,121,169,243,289,343.
\]
It separates unit-circle bijectivity from the full-field condition and directly
verifies the small sporadic cases over \(\F_{q^2}\).  The output terminates
with no mismatches.

\clearpage
\phantomsection

\end{document}